\documentclass[11 pt]{amsart}
\usepackage{amssymb}

\usepackage{mathrsfs}
\usepackage[all]{xy}
\usepackage{graphicx}
\usepackage{wrapfig}
\theoremstyle{plain}
\newtheorem{theorem}{Theorem}
\newtheorem{proposition}{Proposition}
\newtheorem{lemma}{Lemma}

\theoremstyle{remark}

\newtheorem{example}{Example}

\def\d{\operatorname{dist}}

\def\To{\longrightarrow}

\def\sig{\operatorname{sign}}
\def\Lip{\operatorname{Lip}}

\def\L{\mathscr L}

\def\e{\varepsilon}

\def\R{\mathbb{R}}
\def\N{\mathbb{N}}
\def\H{\mathbb{H}}
\def\I{\mathbb{I}}

\def\U{\mathscr{U}}
\def\V{\mathscr{V}}

\def\LX{\Lip^*(X)}
\def\LY{\Lip(Y)}
\def\LE{\Lip E}

\title{Quiz your maths: do the uniformly continuous  functions on the line form a ring?}
\author{F\'elix Cabello S\'anchez and Javier Cabello S\'anchez}
\address{Departamento de Matem\'{a}ticas, UEx, and IMUEx, 06071-Badajoz, Spain}
\email{fcabello@unex.es, coco@unex.es}


\thanks{Supported in part by DGICYT project MTM2016$\cdot$76958$\cdot$C2$\cdot$1$\cdot$P (Spain) and Junta de Extremadura programs GR$\cdot$15152 and  IB$\cdot$16056.}

\begin{document}

\begin{abstract}
The paper deals with the interplay between boundedness, order and ring structures in function lattices on the line and related metric spaces. It is shown that the lattice of all Lipschitz functions on a normed space $E$ is isomorphic to its sublattice of bounded functions if and only if $E$ has dimension one. The lattice of Lipschitz functions on $E$ carries a ``hidden'' $f$-ring structure with a unit, and the same happens to the (larger) lattice of all uniformly continuous functions for a wide variety of metric spaces.

An example of a metric space whose lattice of uniformly continuous functions supports no unital $f$-ring structure is provided. 
\end{abstract}

\maketitle

\markboth{{\it F. and J. Cabello S\'anchez}}{{\it Quiz your maths}}

\section{Introduction}
This paper deals with the interplay between boundedness, order and ring structures in the lattices of uniformly continuous and Lipschitz functions on the real line and other related metric spaces.

\subsection{An esoteric motivation}\label{sec:esoteric}
Let $U$ denote the vector lattice of all uniformly continuous functions on the half-line $\H=[0,\infty)$ and $U^*$ the sublattice of bounded functions.
Although the paper can be read without any reference to the results on homomorphisms quoted below, the truth is that our initial motivation stemmed from the following facts that we can regard as ``empirical data'':

\smallskip

$\bigstar$ Each lattice homomorphism $\phi:U\To\R$ has the form
$$
\phi(g)=k\lim_{\U(n)}\frac{g(2^{c+n}-1)}{2^{c+n}}\quad\quad(g\in U),
$$
where $k$ is the value of $\phi$ at the function $t\longmapsto 1+t$, $\U$ is an ultrafilter on $\N_0$, the nonnegative integers and $c\in[0,1]$. Moreover, $(c,\U)$ and $(d,\V)$ induce the same homomorphism if and only if
$c=1, d=0$ and $\V=1+\U$,  that is, the sets of $\V$ are obtained
by translating those of $\U$ by a unit, or vice-versa (cf. \cite[\S~3]{fine}).

\smallskip

$\bigstar$ Each homomorphism $\phi:U^*\To\R$ has the form
$$
\phi(f)=k\lim_{\U(n)}f(c+n)\quad\quad(f\in U^*),
$$
where $k=\phi(1)$ and $c,\U$ are as before, including the rule about when $(c,\U)$ and $(d,\mathscr V)$ represent the same homomorphism. This is just a description of the Samuel-Smirnov compactification of the half-line and most probably belongs to the topological folklore. The closest results that we have found in the literature are \cite[Theorem 2.1]{turkiyos} and \cite[Lemma 2.1]{japo}.

\smallskip

These statements, that we have slightly edited in order to match with the notation that we shall use along the paper, can be made more precise and provide a quite natural homeomorphism between the spaces of homomorphisms of $U$ and $U^*$ which suggest not only that $U$ and $U^*$ must be very similar lattices but also that one might switch between them sending each $f\in U^*$ to the function defined by
\begin{equation}\label{sending}
t\longmapsto (1+t)f(\log_2(1+t)),
\end{equation}
or, going in the opposite direction, sending $g\in U$ to
\begin{equation}\label{sending2}
s\longmapsto g(2^{s}-1)/{2^{s}}.
\end{equation}

Incidentally, at a certain stage of this research, one of the present authors conjectured that $U^*$ and $U$ are isomorphic lattices, while the other one proved that they aren't... Fortunately, Leung and Tang had already put the matter to rest in \cite[\S 6.6, Lemma~6.37]{l-t}:

\medskip

$\bigstar$ If $X$ and $Y$ are metric spaces for which $U(X)$ and $U^*(Y)$ are order isomorphic (that is, there is a not necessarily linear bijection that  preserves order in both directions) then $U(X)=U^*(X)$ and $X$ is uniformly homeomorphic to $Y$.

\medskip
In our defence we can only add that the booklet \cite{l-t} is very recent.

\subsection{Plan of the paper}
Let us describe the organization of the paper and summarize the main results. 
The remainder of this Section contains some definitions, mainly to fix the notation.

In 
Section~\ref{sec:Lip} we transfer the transformations in (\ref{sending}) and (\ref{sending2}) to an arbitrary Banach space and we study their action on Lipschitz and uniformly continuous functions. It is perhaps a little ironic that while  our initial attempt to produce an isomorphism between $U$ and $U^*$ was doomed to fail, it works fine with  Lipschitz functions. Thus, the lattice of all Lipschitz functions on the (half) line is isomorphic to its sublattice of bounded functions. 
We then show that this result does not generalize to higher dimensions. These results have loose connections with the ideas of \cite[Section~5]{l-t} and \cite[Proposition 1.7.5]{wi}.

Moving in a different direction, and motivated by the fact that for any metric space $X$, the lattice $\Lip(X)$ carries a unital $f$-ring structure, Section~3 explores the possibility that also the lattices $U(X)$ carry such  ``hidden'' $f$-ring structures. It is shown that this is indeed the case for a wide variety of metric spaces, including length and quasiconvex spaces.


The last part of the paper, somewhat in the spirit of Menard's celebrated work on the game of chess \cite{jlb}, 
presents an example of a metric space whose lattice of uniformly continuos functions 
cannot be given any unital $f$-ring structure.


\section*{Notations, conventions}

\subsection{Function lattices}
A function lattice $\L$ on the set $X$ is just a linear subspace of the lattice $\R^X$ which is also a sublattice. All operations on  $\R^X$ are defined pointwise.
If $X$ is a metric space, we denote by $C(X), U(X)$ and $\Lip(X)$ the lattices of continuous, uniformly continuous and Lipschitz functions on $X$, respectively and by 
$C^*(X), U^*(X)$ and $\Lip^*(X)$ the corresponding sublattices of bounded functions.

We denote the distance function by $d$ and by $L(f)$ the Lipschitz constant of $
f:X\To\R$; if $f$ is bounded, then we write $\|f\|_\infty=\sup_{x}|f(x)|$. The norm in $\LX$ is $\|f\|_{\LX}=\|f\|_\infty+L(f)$. The norm in $\Lip(X)$ is defined as $\|f\|_{\Lip(X)}=|f(z)|+L(f)$, where $z$ is a fixed, ``distinguished'' point (the origin, if $X$ is a normed space). The function $\Delta_z(x)=1+d(x,z)$ will be used frequently. When $X$ is a normed space, then $\Delta_0=1+\|\cdot\|$.

\subsection{Homomorphisms}
By a homomorphism of vector lattices we mean a linear map preserving joins and meets (equivalently, absolute values). Given a vector lattice $\mathscr L$,  we denote by $H(\mathscr L)$ the set of all homomorphisms $\phi:\mathscr L\to\mathbb R$.

As $H(\mathscr L)$ is a subset of $\R^\L$ we can equip it with the relative product topology, which is called the pointwise topology in this context: a typical neighbourhood of $\phi$ has the form
$$
\{\psi\in H(\mathscr L): |\phi(f_i)-\psi(f_i)|<\varepsilon \text{ for $1\leq i\leq n$}\},
$$
where $f_1,\dots, f_n\in \mathscr L$ and $\varepsilon >0$. This is the only topology that we will consider on $H(\mathscr L)$.

\subsection{$f$-rings}
These are vector lattices $\mathscr L$ equipped with a product $M:\mathscr L\times\mathscr L\to\mathscr L$ which, in addition to satisfying the usual algebraical requirements (distributive, associative), is a lattice homomorphism in each variable when the other argument is a fixed positive element of $\mathscr L$.
Equivalently, one has
$$
|M(f,g)|=M(|f|,|g|)\quad\quad(f,g\in \mathscr L).
$$
This is not the original definition, but an equivalent condition; see \cite[Theorem~3.15]{johnson}. According to Johnson, the class of $f$-rings was isolated by Birkhoff and Pierce to avoid some disquieting pathologies of the broader class of lattice ordered rings. Note that every function lattice which is closed under pointwise multiplication is an $f$-ring. We refer the reader to Steinberg's \cite[Chapter 3]{steinberg} for a very complete and readable introduction to $f$-rings and in particular for the issue of unitability.

\section{Lipschitz affairs}\label{sec:Lip}

\subsection{The main transformations}
In this Section we graft the maps looming through the results quoted in Section~\ref{sec:esoteric} into an arbitrary Banach space and study their main properties. So, given a Banach space $E$, we consider the following mutually inverse automorphisms of $C(E)$:
\begin{align*}
\Phi f(y)&=(1+\|y\|)f\left( \frac{\log(1+\|y\|)}{\|y\|}\: y	\right), \text{ with } \Phi f(0)= f(0)\\
\Gamma g(x)&=e^{-\|x\|}g\left( \frac{e^{\|x\|}-1}{\|x\|}\:x	\right), \text{ with }(\Gamma g)(0)= g(0).
\end{align*}
We have replaced $2$ by the usual basis of the logarithm in (\ref{sending}) and (\ref{sending2}). Otherwise these maps consist of copies of the maps based on the half-line ``pasted'' along the rays of $E$. Actually we could replace $E$ by any subset closed under homotheties. The following result summarizes the main properties of these maps:

\begin{proposition}\label{prop:Gamma}
With the preceding notations:
\begin{itemize}
\item[(a)] $\Phi$ and $\Gamma$ are mutually inverse automorphisms of $C(E)$.
\item[(b)] $\Gamma$ maps $\Lip(E)$ into $\Lip^*(E)$ and $U(E)$ into $U^*(E)$.
\item[(c)] If $\Delta_0(x)=1+\|x\|$, then $\Gamma(fg/\Delta_0)=\Gamma(f)\Gamma(g)$.
\item[(d)] If $\mathscr L$ is a linear subspace of $C(E)$ containing every function of the form $fg/\Delta_0$, with $f,g\in\mathscr L$, then $\Gamma[\mathscr L]$ is a subring of $C(E)$, which is unital when $\Delta_0\in\mathscr L$.
\item[(e)] Both $\Lip(E)$ and $U(E)$ satisfy the conditions of Part {\rm(d)}.
\end{itemize}
\end{proposition}

\begin{proof}
Consider the function $a:\H\To[1,\infty)$ defined by $a(t)=(e^t-1)/t$ for $t>0$ and $a(0)=1$. Then, if we set $\gamma(y)=a(\|y\|)\:y$ for $y\in E$, we have
$$
\Gamma g(y)=e^{-\|y\|}g\big{(}\gamma(y)\big{)}\quad\quad(g\in C(E), y\in E).
$$
Note that $a$ is an increasing, differentiable, convex function, so for $0\leq s <t$ one has 
$a(t)-a(s)\leq a'(t)(t-s)$. Besides, as $ta(t)=e^t-1$ we have $a(t)+ta'(t)=e^t$ for every $t>0$.
Thus, for any $x,y\in E$ with $\|x\|\leq \|y\|$,
\begin{equation}\label{convex}
\begin{aligned}
\|\gamma(y)-\gamma(x)\|
&\leq
\left\|a(\|y\|)\:y- a(\|y\|)\:x\|+ \|  a(\|y\|)\:x -a(\|x\|)\:x\right\|\\
&\leq a(\|y\|)\|y- x\| + a'(\|y\|)(\|y\|-\|x\|)\|x\|\\
&\leq e^{\|y\|}\|y- x\|.
\end{aligned}
\end{equation}
Let us proceed with the proof.
Part (a) is trivial.
To prove (b) let us first observe that every $g\in U(E)$ is ``Lipschitz for large distances'': for every $\varepsilon>0$ there is a constant $L_\e=L_\e(g)$ such that $|g(x)-g(y)|\leq L_\e\|x-y\|$ whenever $x,y\in E$ are such that $\|x-y\|\geq \e$. This implies that
 $\Delta_0$ is a ``dominant'' function in $U(E)$: for every $g\in U(E)$ there is a constant $C>0$ such that $|g|\leq C\Delta_0$.
Since $\Gamma(\Delta_0)=1_E$ we see that $\Gamma(g)$ is bounded for every $g\in U(E)$.

Now suppose $g$ Lipschitz on $E$.
Clearly, $\Gamma g$ is bounded, with $\|\Gamma(g)\|_\infty\leq \|g\|_{\LE}$. Besides, for $\|y\|\geq \|x\|$ one has, by (\ref{convex}),
\begin{equation}\label{convex2}
\begin{aligned}
\|\Gamma g(y)-\Gamma g(x)\|
&\leq e^{-\|y\|}\| g(\gamma(y))-g(\gamma(x)) \|+ \left(e^{-\|x\|}-e^{-\|y\|}	\right)\| g(\gamma(x)) \|\\
&\leq e^{-\|y\|}L(g)e^{\|y\|}\|y-x\|+
 e^{-\|x\|} \|y-x\|\|g(\gamma(x))\|\\
 &\leq 
 \big{(}L(g)+ \|\Gamma g\|_\infty\big{)}\|y-x\|
\end{aligned}
\end{equation}
so $\Gamma g$ is Lipschitz.

To complete the proof of (b) it suffices to see that $\Gamma$ preserves uniform continuity.
Pick $g\in U(E)$. Take sequences $(x_n)$ and $(y_n)$ in $E$ such that $\|x_n-y_n\|\to 0$. We want to see that
$$
\Gamma g(x_n)-\Gamma g(y_n)= e^{-\|x_n\|}g\left(\gamma(x_n) \right)
-
e^{-\|y_n\|}g\left(\gamma(y_n) \right) \To 0
$$
as $n\to\infty$. This is obvious if $\gamma(x_n)-\gamma(y_n)\longrightarrow 0$, thus, passing to a subsequence if necessary and exchanging $x_n$ and $y_n$ when 
 $\|x_n\|>\|y_n\|$, we may assume and do that:
\begin{itemize}
\item $\|x_n\|\leq\|y_n\|$ for all $n$,
\item $\|\gamma(x_n)-\gamma(y_n)\|\geq \e$ for some $\e>0$ and all $n$.
\end{itemize}
Then, replacing $L(g)$ by $L_\e(g)$ and proceeding as in (\ref{convex2}), one obtains the bound
$
\|\Gamma g(y_n)-\Gamma g(x_n)\|\leq 
 \big{(}L_\e(g)+ \|\Gamma g\|_\infty\big{)}\|y_n-x_n\|,
$
 which is enough.

Part (c) is basically trivial, and (d) follows from (c). Let us check (e) for Lipschitz functions. Take $f,g\in\Lip(E)$, and assume that $|f|,|g|\leq \Delta_0$ and $L(f),L(g)\leq 1$.
Pick $x,y\in E$ such that $\|x\|\leq \|y\|$. One has
\begin{align*}
\frac{f(x)g(x)}{1+\|x\|}&-\frac{f(y)g(y)}{1+\|y\|}\\
&=
\underbrace{\frac{f(x)g(x)}{1+\|x\|}-
\frac{f(x)g(y)}{1+\|x\|}}_{(\star)}+
\underbrace{\frac{f(x)g(y)}{1+\|x\|}
-
\frac{f(x)g(y)}{1+\|y\|}}_{(\star\star)}
+
\underbrace{\frac{f(x)g(y)}{1+\|y\|}
-
\frac{f(y)g(y)}{1+\|y\|}}_{(\star\star\star)}
\end{align*}
The absolute value of the first and third summands cannot exceed $\|x-y\|$. As for the second one observe that for real $0<s<t$, the concavity of the function $s\mapsto 1/s$ and the mean value theorem imply $1/s-1/t\leq s^{-2}(t-s)$, so
$$
|(\star\star)|\leq \frac{|f(x)g(y)|}{(1+\|x\|)^2}|\|y\|-\|x\| |\leq \|x-y\|,
$$
which completes the proof for $\LE$. The corresponding statement about $U(E)$ can be proved analogously; see Lemma~\ref{lem:delta} for a more general result.
\end{proof}

The same argument shows that if $X$ is any metric space and we fix a point $z\in X$, then, for every $f,g\in \Lip(X)$, the function $fg/\Delta_z$
is Lipschitz. Thus, the lattice $\Lip(X)$ is a unital $f$-ring
under the product
$
M_z(f,g)={fg}/\Delta_z.
$
The resulting $f$-ring structures are essentially independent on $z$. Indeed if $y$ is another point of $X$, then 
$\Psi(f)=\Delta_y f/\Delta_z$ defines  a lattice automorphism of $\Lip(X)$
intertwinning $M_z$ and $M_y$ in the sense that
$M_y(\Psi f,\Psi g)=\Psi(M_z(f,g))$ for every $f,g\in\Lip(X)$.

It is worth remarking that, if $X$ is a metric space, then
the lattice $\Lip(X)$ is isomorphic to some $\Lip^*(Y)$, and actually $Y$ can be taken to be $X$, with a new metric which is bounded and uniformly equivalent to the original metric of $X$; see  \cite[Proposition 1.7.5]{wi} and \cite[Section~5]{l-t}.

\subsection{Peculiarities of the line}
Our main ``concrete'' result is the following.

\begin{theorem}\label{th:dim1}
The map $\Gamma:\Lip(\R)\To\Lip^*(\R)$ is a surjective isomorphism of vector lattices.
\end{theorem}

\begin{proof}
We already know that $\Gamma$ takes Lipschitz functions into bounded Lipschitz functions. The point is to prove that it is onto when the underlying metric space is $\R$.
As $\Phi$ is the inverse of $\Gamma$ in $C(\R)$ one only has to prove that $\Phi$ takes bounded Lipschitz functions into Lipschitz functions.

Pick $f\in \Lip^*(\R)$ and set $g=\Phi f$, that is,
$
g(t)=(1+|t|)f\left( \sig(t)\log(1+|t|)	\right).
$
It suffices to see that $g$ is Lipschitz on $[0,\infty)$ and $(-\infty,0]$. Let us check the Lipschitz condition on the ``positive'' half-line.
Take $0\leq s<t$. The function $s\mapsto \log(1+s)$ is concave, so by the mean value theorem
$$
\log(1+t)- \log(1+s)\leq (t-s)/(1+s).
$$
Now,
\begin{align*}
|g(t)&- g(s)|=| (1+t)f\left( \log(1+t)	\right)- (1+s)f\left( \log(1+s)	\right)|\\
&\leq | (1+t)- (1+s)||f\left( \log(1+t)	\right)|+ (1+s)|f\left( \log(1+t)	\right)-f\left( \log(1+s)	\right)|\\
&\leq \|f\|_\infty|t-s|+ (1+s)L(f)(t-s)/(1+s)\\
&\leq \|f\|_{\Lip^*}|t-s|,
\end{align*}
which completes the proof.
\end{proof}

Theorem~\ref{th:dim1} and Proposition~\ref{prop:Gamma} together confirm our initial impression that $U(\R)$ and $U^*(\R)$ are very similar objects and indeed the homomorphism $\Gamma: U(\R) \To U^*(\R)$
is very close to being 
an isomorphism as its range is a unital subring of $U^*(\R)$ which is moreover uniformly dense since it contains all bounded Lipschitz functions.

\subsection{No isomorphism in higher dimensions}
Theorem~\ref{th:dim1} does not generalize to higher dimensions:
easy computations in polar coordinates reveal that there exist $f\in\Lip^*(\R^2)$ for which  $\Phi f$ fails to be Lipschitz.
Thus $\Gamma$ cannot provide an isomorphism between $\Lip(\R^2)$ and $\Lip^*(\R^2)$.
Actually, these lattices are not isomorphic, as we have the following result where we do not assume that the spaces are finite dimensional.

\begin{proposition}\label{prop:sharp} Let $X$ and $Y$ be normed spaces. If $X$ has dimension greater that 1, then the vector lattices $\Lip^*(X)$ and $\Lip(Y)$ are not isomorphic.
\end{proposition}

\begin{proof}
Assume $\Psi: \Lip^*(X)\To\Lip(Y)$ is an isomorphism, where $X$ and $Y$ are normed spaces that we can assume complete since the lattice of  Lipschitz (or bounded Lipschitz) functions on a given metric space is naturally isomorphic to that of its completion.
Then $\Psi$ must have the form
$$
\Psi f (y)=u(y) f(\psi(x))\quad\quad(f\in \Lip^*(X), y\in Y),
$$
where:
\begin{itemize}
\item[(a)] $\psi: Y\To X$ is a homeomorphism.
\item[(b)] $u\in\Lip(Y)$, with $c\Delta_0\leq u\leq C\Delta_0$ for some $c,C>0$.
\item[(c)] The operator $\Psi$ is bounded when $\Lip^*(X)$ is normed with $f\mapsto \|f\|_\infty+ L(f)$ and $\Lip(Y)$ with $g\mapsto |g(0)|+L(g)$.

\item[(d)] $\psi$ is a Lipschitz map.
\item[(e)] The map $\psi^{-1}:X\To Y$ takes bounded sets of $X$ into bounded sets of $Y$ or, equivalently, $\|\psi(y)\|\To\infty$ as $\|y\|\To\infty$.
\end{itemize}

Let us justify this hotchpotch. The functional representation of $\Psi$ and (a) follow, for instance, from \cite[Theorem 4.3 and Proposition 4.4]{l-t}. Another possibility is to observe that if $\mathscr L$ is either $\Lip^*(X)$ or $\Lip(X)$, then the only nonzero lattice homomorphisms $\varphi: \mathscr L\To\R$ that have a countable neighborhood base in $H(\mathscr L)$ are those of the form $\varphi(f)=cf(x)$ for some fixed $c>0$ and $x\in X$. This can be proved as \cite[Lemma~3.3]{g-j-MfM}. 

The second item is obvious: $u=\Psi(1_{X})$ is Lipschitz which gives the upper bound. The lower one follows from the fact that $u$ has to be dominant in $\Lip(Y)$ since $1_X$ is dominant in $\Lip^*(X)$.
Multiplying $\Psi$ by $c^{-1}$ we may and do assume that $c\geq 1$ in the rest of the proof.

The third one easily follows from the closed graph theorem as both $\Lip^*(X)$ and $\Lip(Y)$ are complete in the given norms, and quite clearly $\Psi$ has closed graph.

In the remainder of the proof we set $M=\|\Psi\|$.

To prove that $\psi$ is Lipschitz, take $y,y'\in Y$ and let us  first assume $\|\psi(y)-\psi(y')\|\leq 1$. 
Set $x=\psi(y),
x'=\psi(y')$ and take a nonnegative $f\in\Lip^*(X)$ such that $f(x)=0, f(x')=\|x-x'\|, \|f\|_\infty=\|x-x'\|, L(f)=1$. In particular $\|f\|_{\Lip^*(X)}\leq 2$. 
If $g=\Psi(f)$, then 
$$\|g\|_{\Lip^*(Y)}\leq 2M, \quad g(y)=0,\quad g(y')=u(y')\|x-x'\|\geq (1+\|y'\|)\|x-x'\|,
$$
hence
\begin{equation}\label{crucial}
{({1+\|y'\|})\|\psi(y')-\psi(y)\|}\leq 2M\|y'-y\|\quad\quad(\text{provided } \|\psi(y)-\psi(y')\|\leq 1).
\end{equation}
If $\|\psi(y)-\psi(y')\|>1$, take $n\in\N$ and divide the segment $[y,y']$ using the points
$$
y_i=y+\frac{i}{n}(y'-y)\quad\quad(0\leq i\leq n).
$$
Then $\|y_{i+1}-y_i\|=\|y'-y\|/n$ and since $\psi$ is uniformly continuous on $[y,y']$ we have $\|\psi(y_{i+1})-\psi(y_i)\|\leq 1$ for sufficiently large $n$. Hence,
$$
\|\psi(y')-\psi(y)\|\leq \sum_{1\leq i\leq n}\|\psi(y_{i+1})-\psi(y_i)\|\leq 2M  \sum_{1\leq i\leq n}\|y_{i+1}-y_i\|=2M\|y'-y\|
$$
and $\psi$ is globally Lipschitz, with Lipschitz constant $2M$.

Part (e) is obvious for any homeomorphism $\psi:Y\To X$ if $Y$ (hence $X$) is finite dimensional, using local compactness. For infinite dimensional spaces the argument is a bit trickier and is deferred to the end of this Section. 

\smallskip

We will reach a contradiction using a gorgeous argument kindly provided by the referee. (Our original approach, based on the non contractibility of spheres in finite dimensional spaces, does not apply to the infinite dimensional ones, cf. \cite{b-s}.)
First of all, applying a translation if necessary, we may assume and do that $\psi(0)=0$.
Pick any real (large) number $L>0$ and be $R$ sufficiently large to guarantee that $\|\psi(y)\|>L$ for $\|y\|\geq R$.
Set
$$
S=\{y: \|y\|=2R\},\quad  U=\{y: \|y\|<2R\},\quad V=\{y: \|y\|>2R\}
$$
Now, pick any $a\in S$ and consider the ray $\{-t\psi(a): t\geq 0\}$. Then $-t\psi(a)\in \psi[U]$ for $t$ small, while  $-t\psi(a)\in \psi[V]$ for $t$ large enough and since $\psi[S]=\partial\:\psi[U]=\partial\:\psi[V]$ we conclude that for some $s>0$ one has $-s\psi(a)\in\psi[S]$. Fix this $s$ and take $b\in S$ such that $\psi(b)=-s\psi(a)$.

As the points $\psi(a)$ and $\psi(b)$ lie on opposite directions and $\|a\|=\|b\|=2R>R$ we have that $\|\psi(a)\|, \|\psi(b)\| >L$, hence $\|\psi(a)-\psi(b)\| >2L$.

Let $P$ be a two-dimensional subspace of $Y$ containing the points $a$ and $b$. The Banach-Mazur distance between $P$ and the Euclidean plane is bounded by $\sqrt{2}$ and so there is an absolute constant $C_0<\infty$ and points $(y_i)_{0\leq i\leq n}$ in $S$ such that:
\begin{itemize}
\item $a=y_0, b=y_n$.
\item The polygonal path $\bigcup_{1\leq i\leq n} [y_{i-1}, y_i]$ does not meet the ball of radius $R$.
\item $\sum_{1\leq i\leq n} \|y_{i-1}-y_i\|\leq C_0R$.
\end{itemize}
Actually $C_0=4$ suffices by  a  classical result of Go\l\c ab \cite{golab} that the reader can consult in \cite[Theorem 4I, p 29]{schaeffer}.
Since the set $\bigcup_{1\leq i\leq n} [y_{i-1}, y_i]$ is compact and $\psi$ is uniformly continuous on it, by inserting extra points in each segment if necessary, we may
also assume that  $\|\psi(y_{i-1})-\phi(y_i)\|\leq 1$ for $1\leq i\leq n$. Now, we can use (\ref{crucial}) to get
$$
2L\leq \|\psi(a)-\psi(b)\|\leq \sum_{1\leq i\leq n} \|\psi(y_{i-1})-\psi(y_i)\| 
\leq 2M \sum_{1\leq i\leq n} \frac{\|y_{i-1}-y_i\|}{1+\|y_i\|}\leq \frac{8 M R}{1+R},
$$
which is a contradiction for sufficiently large $L$.
\end{proof}

\subsection{The missing piece}
This Section contains the proof of the statement (e) upholding the proof of the preceding Proposition for general (infinite dimensional) Banach spaces where local compactness cannot be invoked.

By the way, if $X=c_0$, one can construct a homeomorphism $\phi:X\To X$ which is moreover a Lipschitz map, and such that $\phi$ maps a line of $X$ into the unit ball of $X$, so $\phi^{-1}$ clearly maps the unit ball into an unbounded set. So one cannot expect to derive ``boundedness'' of the inverse from any ``general'' result.

Throughout the  Section we denote by $rB_X$ and $rS_X$ the closed ball of radius $r$, centered at the origin of $X$, and the corresponding sphere, respectively; and similarly for $Y$.

\begin{lemma}\label{lem:aswell}
Let $\Psi: \Lip^*(X)\To\Lip(Y)$ be an isomorphism of vector lattices. Suppose that $\Psi f(y)=u(y)f(\psi(y))$ for every $f\in\LX$. Then the map $\Psi_0: \Lip^*(X)\To\Lip(Y)$ defined by $\Psi_0 f(y)=(1+\|y\|)f(\psi(y))$ is an isomorphism as well.
\end{lemma}

\begin{proof}
If $\Delta_0(y)=1+\|y\|$, then we know that $c\Delta_0\leq u\leq C\Delta_0$ for some $C,c>0$. We have $u=\Psi(1_X)$ and since $\Delta_0\in \LY$ there is $u_0\in\LX$ such that $\Delta_0=\Psi(u_0)$. Clearly, $c\leq u_0\leq C$, so $1/u_0$ is Lipschitz on $X$ and multiplication by $u_0$ defines an automorphism of $\LX$. 
Thus, the map $\Psi_0: \Lip^*(X)\To\Lip(Y)$ defined by
$\Psi_0(f)=\Psi(u_0 f)$
 is a lattice  isomorphism. But
$$\Psi_0 f(y)=u(y)u_0(\psi(y))f(\psi(y))=(\Psi u_0)(y)f(\psi(y))=\Delta_0(y)f(\psi(y)),
$$
 which is enough.
\end{proof}

\begin{lemma}Let $\psi:Y\To X$ be a homeomorphism. Suppose that for some $u\in\LY$ the map $\Psi:\Lip^*(X)\to\Lip(Y)$ defined by $\Psi f(y)=u(y)f(\psi(y))$ is an isomorphism. Then $\psi^{-1}:X\to Y$ maps bounded sets into bounded sets. 
\end{lemma}

\begin{proof}
We may assume that $\psi(0)=0$, using a translation if necessary, and then that 
$u=\Delta_0$, according to the preceding Lemma. Hence, 
$$\Psi f(y)=(1+\|y\|)f(\psi(y))\quad\mathrm{and}\quad 
\Psi^{-1}g(x)=\frac{g(\psi^{-1}(x))}{1+\|\psi^{-1}(x)\|}$$
for $f\in\Lip^*(X), g\in\Lip(Y)$.

We define a ``special'' Lipschitz function on $Y$ as follows. We take any Lipschitz $h:\H\To\H$ such that
$$
h(t)=\begin{cases}
0 & \text{if $t=0$ or $t=2^k$ for even $k$;}\\
1+2^k & \text{if $t=2^k$ for odd $k$.}
\end{cases}
$$
(A piecewise linear one would suffice.) Now, put $g(y)=h(\|y\|)$. Then $g$ is Lipschitz, agrees with $\Delta_0$ on the spheres of radii $2^k$ when $k$ is odd and vanishes on every sphere of radius $2^k$ for even $k$ and at the origin. If $f=\Psi^{-1}g$, then $f$ is Lipschitz and bounded, vanishes on $\psi[2^k S_Y]$ for even $k$ and the origin and $f=1$ on $\psi[2^k S_Y]$ for odd $k$. Thus, if $r=1/L(f)$, then 
$\d(0,\psi[S_Y])\geq r$ and $\d(\psi[2^{k}S_Y],\psi[2^{k+1}S_Y])\geq r$ for $k\in \N_0$

Let us see that $krB_X\subset \psi[2^kB_Y]$ by induction  on $k\in\N_0$, thus ending the proof. The initial step ($k=0$) is obvious; to check the inductive step notice that as $X$ is connected, if $A, B$ are disjoint subsets of $X$, then $\d(A,B)=
\d(\partial A,\partial B)$. Now assume that $krB_X\subset \psi[2^kB_Y]$. We have
\begin{align*}
\d&(\psi[2^{k+1}B_Y]^c, \psi[2^kB_Y])\\
&= \d\big{(}\partial (\psi[2^{k+1}B_Y]^c),\partial (\psi[2^{k}B_Y])\big{)}=
\d(\psi[2^{k}S_Y],\psi[2^{k+1}S_Y])\geq r,
\end{align*}
hence $\d(krB_X, \psi[2^{k+1}B_Y]^c)\geq r$, which implies that 
$(k+1)rB_X\subset \psi[2^{k+1}B_Y]$ so we are done. 
\end{proof}

\section{Ring structures on $U(X)$}

In this final Section  we study $f$-ring structures on the lattice $U(X)$ when $X$ is a metric space.
We believed, for some time, that a possible reason preventing the lattices $U$ and $U^*$ from being isomorphic was that the former should not admit a unital $f$-ring structure. However, it is an obvious consequence of Proposition~\ref{prop:Gamma} that, when $X$ is a normed space, the product $f,g\longmapsto fg/\Delta_0$ makes $U(X)$ into an $f$-ring whose unit is precisely $\Delta_0$. 

\subsection{The secret life of $U(X)$ as a unital $f$-ring}
The following result implies that the same happens for a large classe of metric spaces $X$, including length spaces and quasiconvex spaces (cf. \cite{ga-ja}):

\begin{lemma}\label{lem:delta}
Let $X$ be a metric space and suppose there is a uniformly continuous function $\Delta:X\To(0,\infty)$ such that for each $f\in U(X)$ the ratio $|f|/\Delta$ is bounded. Then, given $f,g\in U(X)$, the function $fg/\Delta$ is uniformly continuous and the multiplication $M:U(X)\times U(X)\To U(X)$ defined as $M(f,g)=fg/\Delta$ makes $U(X)$ into an $f$-ring with unit $\Delta$.
\end{lemma}

\begin{proof}
We may assume that $|f|, |g| \leq \Delta$.
Fix $\e>0$ and take $\delta>0$ so that $|f(x)-f(y)|,|g(x)-g(y)|, |\Delta(x)-\Delta(y)|\leq \e$ for $d(x,y)\leq\delta$. Pick such $x,y\in X$ and assume 
 $\Delta(x)\leq \Delta(y)$. Then
\begin{align*}
\frac{f(x)g(x)}{\Delta(x)}&-\frac{f(y)g(y)}{\Delta(y)}\\
&=
\underbrace{\frac{f(x)g(x)}{\Delta(x)}-
\frac{f(x)g(y)}{\Delta(x)}}_{(\star)}+
\underbrace{\frac{f(x)g(y)}{\Delta(x)}
-
\frac{f(x)g(y)}{\Delta(y)}}_{(\star\star)}
+
\underbrace{\frac{f(x)g(y)}{\Delta(y)}
-
\frac{f(y)g(y)}{\Delta(y)}}_{(\star\star\star)}.
\end{align*}
Note that $|(\star)|,|(\star\star\star)|\leq\e$. As for the second summand we have 
$$
|(\star\star)|\leq \frac{|f(x)g(y)|}{\Delta(x)^2}\left(\Delta(y)-\Delta(x)\right)\leq \e,
$$
which completes the proof.
\end{proof}

We hasten to remark that
making a lattice into an $f$-ring is a relatively easy task. Indeed, suppose $\mathscr L$ is a vector lattice. Choose a nonzero homomorphism $\phi:\mathscr L\To \R$, fix some positive $h\in\mathscr L$ and set
$$
M(f, g)= \phi(f)\phi(g)h.
$$
Of course the resulting structure will not have a unit unless the linear space underlying $\mathscr L$ has dimension one.

\subsection{The rift between product and order}
We will take our leave of the reader by showing that there are very natural lattices that cannot be given any unital $f$-ring structure.

As one can imagine, we need to investigate a bit the form of the possible $f$-ring products on function lattices, which involves some technicalities. Let us begin with the following ``density'' result.

\begin{lemma}\label{lem:dens}
Suppose $\mathscr L$ is a function lattice on $X$ which contains a strictly positive function. Then the set $\{c\delta_x: 0\leq c<\infty, x\in X\}$ is dense in $H(\L)$.
\end{lemma}

\begin{proof}
Let $\phi$ be a nonzero homomorphism and $h\in\L$ a strictly positive function such that $\phi(h)=1$. We will prove that $\phi$ can be approximated by homomorphisms of the form $\delta_x/h(x)$. This amounts to check that given finitely many $f_i\in\L$ and $\e>0$ there is $x\in X$ such that
$$
\left|\phi(f_i)-\frac{f_i(x)}{h(x)}\right|<\e\quad\quad(1\leq i\leq n).
$$
If we assume that no such an $x$ exists, then, letting $c_i=\phi(f_i)$, we have
$
\bigvee_{1\leq i\leq n}|c_i h- f_i|\geq \e h.
$ 
But
$$
\phi\left(\bigvee_{1\leq i\leq n}|c_i h - f_i|\right)
=
\bigvee_{1\leq i\leq n}\phi(|c_i h- f_i|)
=
\bigvee_{1\leq i\leq n}|c_i \phi(h)- \phi(f_i)| =0,
$$ 
while $\phi(\e h)=\e$, a contradiction.
\end{proof}

\begin{lemma}\label{lem:bit}
Let $\L$ be a function lattice on $X$ containing a strictly positive function and having the following properties:
\begin{itemize}
\item[(a)] For every $x\in X$ there is $f\in\L$, depending on $x$, such that whenever $\phi\in H(\L)$ vanishes on $f$, then $\phi=c\delta_x$ for some $c\geq 0$.
\item[(b)] If $\phi_1,\phi_2\in H(\L)$ and $\phi_1+\phi_2\in H(\L)$, then $\phi_1$ and $\phi_2$ are linearly dependent.
\end{itemize}
Then every unital $f$-ring structure on $\L$ has the form $M(f,g)=fg/u$ for some strictly positive $u\in\L$ which is necessarily the unit of $M$.
\end{lemma}

\begin{proof}
Let $M:\L\times\L\to\L$ be a bilinear map giving an $f$-ring structure.
We want to see that there exist two mappings $\Lambda,\Sigma:H(\L)\To H(\L)$ such that 
\begin{equation}\label{phiM}
\phi(M(f,g))=\Lambda[\phi](f) \Sigma[\phi](g)\quad\quad(f,g\in\L, \phi\in H(\L)).
\end{equation}
The notation means that $\Lambda[\phi]$ and $\Sigma[\phi]$  are homomorphisms depending on $\phi$ acting on $f$ and $g$, respectively.

If we fix a positive $f\in \L$ and $\phi\in H(\L)$, then the map $g\in\L\mapsto \phi(M(f,g))\in\R$ is a homomorphism depending on $f$ and $\phi$. Denoting it by $\sigma[f,\phi]$ we have 
$$
\phi(M(f,g))=\sigma[f,\phi](g)\quad\quad(f,g\in\L, \phi\in H(\L)).
$$
Since $\sigma[f,\phi]$ is additive in the first variable in the sense that if $f_1,f_2$ are positive in $\L$, then $\sigma[f_1+f_2,\phi]= \sigma[f_1,\phi]+\sigma[f_2,\phi]$, the hypothesis (b) implies that there exist a homomorphism $\Sigma[\phi]$ and a linear map $\Lambda[\phi]:\L\to\R$ such that $\sigma[f,\phi]= \Lambda[\phi](f)\Sigma[\phi]$, that is,
$$
\phi(M(f,g))=\Lambda[\phi](f) \Sigma[\phi](g)\quad\quad(f,g\in\L, \phi\in H(\L)).
$$
To be true the preceding formula is proven only for $f\geq 0$. The general case follows by decomposing $f$ into its positive and negative parts.

Now, considering $\phi$ fixed, either $\Sigma[\phi](g)=0$ for all $g\in\L$ and so $M=0$, which cannot be since $M$ has a unit, or taking some $g$ such that $\Sigma[\phi](g)\neq 0$ we see that the linear map $\Lambda[\phi]$ is also a homomorphism, and we have proved (\ref{phiM}).

Let $u\in\L$ be the unit of $M$, so that $f=M(f,u)=M(u,f)$ for all $f\in\L$ and suppose $\phi\neq 0$. Taking $f\in \L$ such that $\phi(f)\neq 0$ we have
$$
0\neq \phi(f)=\begin{cases}
\phi(M(f,u))=\Lambda[\phi](f)\Sigma[\phi](u),\\
\phi(M(u,f ))=\Lambda[\phi](u)\Sigma[\phi](f)\end{cases}
$$
so $\Sigma[\phi](u), \Lambda[\phi](u) \neq 0$.
To end, pick $x\in X$ and take $f$ as in (a). We then have
$$
0=f(x)= \begin{cases} \delta_x(M(f,u))=\Lambda[\delta_x](f) \Sigma[\delta_x](u)\\
\delta_x(M(u,f))=\Lambda[\delta_x](u) \Sigma[\delta_x](f),
\end{cases}
$$
and since  $\Sigma[\delta_x](u), \Lambda[\delta_x](u)\neq 0$ we see that $\Lambda[\delta_x]=c(x)\delta_x$ and $\Sigma[\delta_x]=d(x)\delta_x$, hence
$$
M(f,g)(x)=c(x)f(x)d(x)g(x)\quad\quad(f,g\in \L, x\in X),
$$
and taking $u(x)=1/(c(x)d(x))$ we are done.
\end{proof}

\begin{example}
Let $\mathbb I=[0,1]$ be the unit interval and consider the following distance on $\I\times\N$:
$$
d((s,n),(t,k))=\begin{cases}
|s-t| & \text{if $n=k$}\\
1 & \text{otherwise}.
\end{cases}
$$
Then the lattice $U(\I\times \N)$ has no unital $f$-ring structure.
\end{example}

\begin{proof}
The lattice $U(\I\times\N)$ has the relevant properties of the preceding Lemma just because each homomorphism $\phi: U(\I\times\N)\To \R$ is a multiple of the evaluation at some point of $\I\times\N$. Indeed, let $\phi: U(\I\times\N)\To \R$ be a nonzero homomorphism and take a strictly positive $h\in U(\I\times\N)$ such that $\phi(h)=1$. On account of Lemma~\ref{lem:dens} and its proof there is an ultrafilter $\U$ on $\I\times\N$ such that
$$
\phi=\lim_{\U(t,n)}\frac{\delta_{(t,n)}}{h(t,n)}.
$$
Let $\V$ and $\mathscr W$ be the ultrafilters obtained by projecting $\U$ onto the first and second coordinates of $\I\times\N$, respectively. It is clear that $\V$ converges to a point $t^*\in \I$, by compactness. We claim that $\mathscr W$ is fixed: otherwise take $M_n=\max_{s\in\I}h(s,n)$ and define $f:\I\times\N\to\R$ letting $f(t,n)=n\cdot M_n$. Then
$$
\phi(f)=\lim_{\U(t,n)}\frac{f(t,n)}{h(t,n)}= \lim_{\U(t,n)}\frac{nM_n}{h(t,n)}\geq 
\lim_{\U(t,n)} n= \lim_{\mathscr W(n)} n=\infty,
$$
an absurd. Hence $\mathscr W$ converges to some $n^*\in \N$ and it follows from the basics on filter convergence (\cite[Chapter~4,\S 12]{wi}) that $\U$ converges to $(t^*, n^*)$ and so
$$
\phi=\frac{\delta_{(t^*,n^*)}}{h(t^*,n^*)},
$$
as desired.

Now suppose $M$ is a unital $f$-ring structure on $U(\I\times\N)$. Then, by Lemma~\ref{lem:bit}, one has 
$$
M(f,g)={fg}/{u},
$$
where $u>0$ is the unit of $M$. Let us consider the functions $f,g:\I\times\N\to\R$ defined by $f(t,n)=t$ and $g(t,n)=n\cdot u(1,n)$. Then $M(f,g)(0,n)=0$ for all $n\in\N$, while $M(f,g)(1,n)=n$ which prevents $M(f,g)$ from being uniformly continuous on $\I\times \N$ since the family $t\mapsto M(f,g)(t,n)$ is not equicontinuous.
\end{proof}

Note that $\I\times \N$ with the given metric is uniformly homeomorphic to a subset of $\H$, namely to $\bigcup_{n\geq 1}[2n-1,2n]$. This example is used, with a different purpose, in \cite[Example~1]{ga-ja}.
\medskip

Thus, while some lattices $U(X)$ are straight rings under the pointwise product, there exist metric spaces $X$ for which $U(X)$ is not closed for pointwise multiplication yet it supports a ``hidden'' unital $f$-ring structure and there are other metric spaces for which $U(X)$ cannot be given such structures at all.

Even the most favourable case where $U(X)$ is a ring with the pointwise product presents various shades ranging from Atsuji spaces (all continuous functions are uniformly continuous) to Bourbaki bounded spaces (all uniformly continuous functions are bounded); see \cite{javi} for some ``intermediate'' examples. We can refer the interested reader to the recent paper \cite{b-g-m} for information on this line of research, proper references and many other things.

\subsection*{Open questions}
Let us mention a few questions that arise naturally from the content of the paper. First, it is unclear for which metric spaces $X$ the lattices $\Lip(X)$ and $\Lip^*(X)$ are isomorphic, either as vector lattices or as ordered sets. We have got just one example where such isomorphism exists and it is a rather ``small'' space. It would be interesting to know what happens to ``larger'' spaces, in particular to Urysohn universal space; see \cite{husek}.

Despite our efforts we have been unable to isolate any property, defined for vector lattices, distinguishing $U$ from $U^*$. In a different direction we do not known whether the existence of a ``dominant'' function in $U(X)$ guarantees that $\Lip(X)$ is uniformly dense in $U(X)$. The converse is nearly obvious.

\section*{Acknowledgements}
It is a pleasure to thank the referee for the careful reading of the original \LaTeX-script. He/she provided the sharp version of Proposition~\ref{prop:sharp} that appears in the text and contributed to a more logical organization of the paper.

The content of this paper was presented in a session of the ``Colloquium of the Department of Mathematical Analysis'' held at the Universidad Complutense (Madrid, Spain) in  November 27, 2018.
We thank the organizers, especially Nacho Villanueva, for their customary warm hospitality.

\bibliographystyle{amsplain}

\end{document}